\documentclass[12pt]{amsart}
\usepackage[]{amsmath, amsthm, amsfonts, verbatim, amssymb, tikz,inputenc}

\usepackage{amsmath}
\usepackage{amsthm}
\usepackage{amssymb}
\usepackage{amsfonts}
\usepackage{hyperref}
\usepackage{rotating}
\usepackage{amssymb}
\usepackage{epstopdf}
\usepackage{pifont}
\usepackage{amsmath}
\usepackage{enumerate}
\usepackage{graphicx}
\usepackage{color}
\usepackage{tikz-cd}
\usepackage{mathtools}

\DeclareMathAlphabet{\mathpzc}{OT1}{pzc}{m}{it}

\usepackage{tikz}
\usetikzlibrary{arrows}

\newtheorem{theorem}{Theorem}[section]
\newtheorem{lemma}[theorem]{Lemma}
\newtheorem{proposition}[theorem]{Proposition}
\newtheorem{corollary}[theorem]{Corollary}
\newtheorem{conjecture}[theorem]{Conjecture}
\newtheorem{definition}[theorem]{Definition}

\theoremstyle{remark}
\newtheorem{remark}[theorem]{Remark}

\newcommand{\bC}{\mathbb{C}}

\newcommand{\bN}{\mathbb{N}}
\newcommand{\bQ}{\mathbb{Q}}
\newcommand{\bP}{\mathbb{P}}

\newcommand{\bZ}{\mathbb{Z}}

\newcommand{\cO}{\mathcal{O}}

\newcommand{\Exc}{{\rm Exc}}

\newcommand{\Img}{{\rm Im}}

\newcommand{\Supp}{{\rm Supp}}
\newcommand{\Sing}{{\rm Sing}}

\def\<{\langle}
\def\>{\rangle}

\title{Generalized nonvanishing conjecture and Iitaka conjecture}
\author{Chi-Kang Chang}
\address{Department of Mathematics, National Taiwan University, Taipei, 106, Taiwan}
\email{d08221001@ntu.edu.tw}

\begin{document}
\maketitle
\begin{abstract}
In this article, we will prove the Generalized Nonvanishing Conjecture holds for threefolds with either $\kappa> 0$ or $q>0$. As a result, we can prove the Iitaka conjecture $C_{n,m}$ holds for $n=7$ if the source space has non-negative Kodaira dimension, if the general fibre has positive Kodaira dimension, or if the base space is not threefold with $\kappa=q=0$. In particular, $C^-_{n,m}$ holds if $n\leq 7$.
\end{abstract}

\section{Introduction}
Given an algebraic fibre space $f:X\rightarrow Y$ between (smooth) projective varieties over a field with characteristic 0, a important question in birational geometry is how to compare birational invariants of $X$, $Y$, and the general fibre $F$ of $f$. One of the most important birational invariants is the Kodaira dimension, denoted by $\kappa(X)$, which is the Iitaka dimension of the canonical divisor. About the Kodaira dimension, a well-known classical conjecture, named the Iitaka Conjecture, is formulated as follows:

\begin{conjecture}[Iitaka Conjecture]
Let $f:X\rightarrow Y$ be an algebraic fibre space between smooth projective varieties, and $F$ be a very general fibre of $f$. 
\begin{enumerate}
    \item $(C_{n,m})$ If $\dim X = n$ and $\dim Y= m$, then we have $$\kappa(X)\geq \kappa(F)+\kappa(Y);$$
    \item $(C^-_{n,m})$ If $\dim X = n$, $\dim Y= m$, and $\kappa(X)\geq 0$, then we have $$\kappa(X)\geq \kappa(F)+\kappa(Y).$$
\end{enumerate}
\end{conjecture}

Obviously, $C_{n,m}^-$ is a weaker version of $C_{n,m}$. Although the general case of $C_{n,m}$ remains unresolved, many special cases have been proven:

\begin{itemize}
    \item $\dim Y \leq 2$ (\cite{K82}, \cite{C});
    \item $\dim X \leq 6$ (\cite{B});
    \item $Y$ is of general type (\cite[Theorem 1.7]{F17},  \cite[Corollary IV]{V});
    \item $F$ is of general type (\cite[Corollary 4.3.5]{F20});
    \item $F$ has a good minimal model (\cite{K85}, \cite{H}). In particular, if $\dim F\leq 3$ (\cite{Kaw2});
    \item $Y$ has maximal Albanese dimension (\cite{CP}, \cite{HPS});
    \item $f$ is a smooth fibration, and $F$ is either a curve or of general type (\cite{PS}).
\end{itemize}
Note that the case $Y$ has maximal Albanese dimension recovered the case $\dim Y =1$, as every smooth projective curve with non-negative Kodaira dimension has maximal Albanese dimension (cf. \cite[Section 4.5.1]{F20}). 

Moreover, there is another classical conjecture, which is called Ueno Conjecture K. This conjecture is closely related to the Iitaka conjecture, as Birkar mentioned in \cite{B}. The statement is expressed as follows:
\begin{conjecture}[Ueno Conjecture K]
Let $X$ be a smooth projective variety with $\kappa(X)=0$, and $\alpha:X\rightarrow A$ be the Albanese map of $X$. Then we have:
\begin{enumerate}
    \item $\alpha$ is surjective;
    \item $\kappa(F)=0$ for the general fibre $F$ of $f$;
    \item There is an \'etale cover $A'\rightarrow A$ such that $X \times_A A'$ is birational to $F \times A'$ over A.
\end{enumerate}
\end{conjecture}
The first part of Ueno Conjecture K is proved by Kawamata in \cite[Theorem 1]{K81}. Moreover, Kawamata proved $\alpha$ is an algebraic fibre space. Also, the work of \cite{CP} and \cite{HPS} implies the second part of the Ueno conjecture is true. 

On the other hand, in Birkar's proof of the Iitaka Conjecture for $\dim X \leq 6$, one of the key ingredients is the canonical bundle formula proved by Fujino and Mori in \cite[Theorem 4.5]{FM}. According to this formula, for an algebraic fibre space $f:X\rightarrow Y$ such that the general fibre has Kodaira dimension 0, there is a good "birational modification" $f':X'\rightarrow Y'$ of $f$ such that the moduli $\bQ$-divisor of $f'$ is nef, and the discriminant $\bQ$-divisor of $f'$ is effective and klt. 

Recently, two interesting conjectures have been introduced by Lazi\'c and Peternell in \cite{LP}, which are called the Generalized Nonvanishing Conjecture and the Generalized Abundance Conjecture. These conjectures are closely related to the generalized pair and the canonical bundle formula for algebraic fibre spaces:

\begin{conjecture}
[Generalized Nonvanishing Conjecture]Let $(X,\Delta)$ be a klt pair of a normal projective variety $X$ such that $K_X+\Delta$ is pseudo-effective. Let $L$ be a nef $\bQ$-divisor on $X$. Then for every rational number $t \geq 0$, the numerical class of the divisor $K_X + \Delta + tL$ belongs to the effective cone, that is, there exists an effective $\bQ$-Cartier divisor $D$ such that $K_X+\Delta+tL\equiv D$.
\end{conjecture}

\begin{conjecture}
[Generalized Abundance Conjecture]Let $(X,\Delta)$ be a klt pair of a normal projective variety $X$ such that $K_X+\Delta$ is pseudo-effective. Let $L$ be a nef Cartier divisor on $X$ such that $K_X+\Delta+L$ is nef. Then there is a semiample $\bQ$-Cartier divisor divisor $M$ satisfying $K_X + \Delta + L\equiv M$.
\end{conjecture}

In \cite[Corollary C and D]{LP}, Lazi\'c and Peternell proved these 2 conjectures when either $\dim X \leq 2$, or $\dim X =3$ with $\nu(K_X+\Delta)\geq 1$. For varieties of higher dimensions, Chaudhuri, in \cite{C23}, proved that the Generalized Abundance Conjecture holds if either the nef dimension (see Definition \ref{nr}) $n(K_X+\Delta+L)\leq 2$, or $n(K_X+\Delta+L)=3$ with $\kappa(K_X+\Delta)>0$. In this article, we will prove some other special cases of the Generalized Nonvanishing Conjecture. Here we will use the following notations:
\begin{itemize}
    \item The Generalized Nonvanishing conjecture in dimension $d$ with $\kappa(X,K_X+\Delta)=k$, $q(X)=q$, and nef divisor $L$ with nef dimension $n$ by $N_{d,k,q,n}$;
    \item The Generalized Nonvanishing conjecture in dimension $d$ with $\kappa(X,K_X+\Delta)=k$ and $q(X)=q$ by $N_{d,k,q}$ (that is, the Generalized Nonvanishing Conjecture for nef divisor of arbitrary nef dimension);
\end{itemize}    

We will prove the following special cases of the Generalized Nonvanishing Conjecture:

\begin{theorem}\label{nvc}
We have the following:
    \begin{enumerate}
        \item $N_{d,k,q,1}$ is true if $k\geq 0$;
        \item $N_{d,k,q,2}$ is true if $k\geq 0$ and $d\leq 5$;
        \item $N_{3,k,q,3}$ is true if $q> 0$;
        \item $N_{d,0,q,n}$ is true if $n=d$ and $q\geq d-2$.
    \end{enumerate}
\end{theorem}

The following corollary immediately follows from \cite[Corollary D]{LP} and Theorem \ref{nvc}:

\begin{corollary}\label{nvc3d}
    $N_{3,k,q,n}$ is true unless $k=0,q=0$, and $n=3$. In particular, $N_{3,k,q}$ is true if either $k>0$ or $q>0$. 
\end{corollary}

Using the Generalized Nonvanishing Conjecture, we can prove the following result concerning algebraic fiber spaces with general fibers having Kodaira dimension zero, through small modifications to the proof in \cite{B}:

\begin{theorem}\label{basenvt}
Let $f:X\rightarrow Y$ be an algebraic fibre space between smooth projective varieties with general fibre $F$. Assume $\kappa(F)=0$ and $N_{d,k,q}$ holds for $d=\dim Y$, $k=\kappa(Y)$, and $q=q(Y)$, then we have $\kappa(X)\geq \kappa(Y)$.
\end{theorem}

In particular, $C_{n,3}$ holds if $\kappa(F)=0$ and either $\kappa(Y)\geq 1$, or $\kappa(Y)=0$ but $q(Y)>0$. As a corollary, we have the following:

\begin{theorem}\label{iitaka7}
Let $f:X\rightarrow Y$ be an algebraic fibre space between smooth projective varieties, and $F$ be a general fibre of $f$. Let $n=\dim X$ and $m=\dim Y$. If $n \leq 7$, then $C_{n,m}$ holds unless $n=7$, $m=3$, $\kappa(F)=\kappa(Y)=0$, and $q(X)=q(F)=q(Y)=\tilde{q}(Y)=0$ (here $ \tilde{q}$ is the augmented irregularity, see Definition \ref{aug}). In particular, $C_{n,m}^-$ holds if $n \leq 7$.
\end{theorem}

The structure of this article is as follows: Section 2 is the preliminary section. In Section 3, we will explore the Generalized Nonvanishing Conjecture for nef divisors $L$ with small nef dimension, by using the nef reduction map of $L$. Following that, in Section 4, we will discuss the Generalized Nonvanishing Conjecture for nef divisors $L$ with maximal nef dimension, employing the Albanese map and the Generalized Abundance in lower dimensions. Theorem \ref{nvc} immediately follows from the results in these two sections. Finally, in Section 5, we will prove Theorem \ref{basenvt} and use it to prove Theorem \ref{iitaka7}.\\

\textbf{Acknowledgments:} Part of this work is done when the author visits the University of Tokyo. The author would like to thank Jungkai Alfred Chen and Yoshinori Gongyo for discussing this question with him and for their helpful advice and encouragement. The author also would like to thank Ching-Jui Lai, Keiji Oguiso, and Mihnea Popa for answering his questions. The author would like  to thank Iacopo Brivio, Jheng-Jie Chen, and Hsueh-Yung Lin for their helpful suggestions and encouragement. The author is supported by the PhD student’s scholarship of National Taiwan University, and by the scholarship of National Science and Technology Council for PhD students to study abroad. The author would like to thank the University of Tokyo for the warm hospitality during his visit.

\section{Preliminaries}
In this article, we work over $\bC$. An {\em algebraic fibre spece} is a surjective projective morphism $f:X\rightarrow Y$ with connected fibre. Over fields with characteristic 0, this condition is equivalent to $f_*\cO_X=\cO_Y$. A $\bQ$-divisor $D$ is called $\bQ$-Cartier if there is a positive integer $m$ such that $mD$ is a Cartier divisor. A variety $X$ is called {\em $\bQ$-factorial} if every divisor on $X$ is $\bQ$-Cartier. A log resolution of a pair $(X,\Delta)$ of normal projective variety is a resolution of singularity $\pi:X'\rightarrow X$ such that $\pi^{-1}(\Sing X\cup \Supp\Delta)$ is a simple normal crossing divisor. A pair $(X,\Delta)$ is called {\em smooth} if $X$ is smooth, and $\Supp\Delta$ is simple normal crossing. Also, through this article we will use the following definitions:
\begin{definition}
    Let $D$ be a $\bQ$-Cartier $\bQ$-divisor. We say $D$ is $\bQ$-effective (resp. numerically $\bQ$-effective, numerically semiample) if there is a positive integer $m$ such that $mD$ is linear equivalent to an effective divisor (resp. numerically equivalent to an effective divisor,  numerically equivalent to a base point free  divisor)
\end{definition}

In the following of this section, we will give a brief review of the definition and basic properties of Kodaira dimension, numerical dimension, nef reduction, nef dimension, and abundant divisors. Further details on these topics can be found in \cite{A}, \cite{BCE}, \cite{BP}, \cite{Laz}, \cite{N}, and \cite{T}. 

\begin{definition}
    Let $X$ be a projective variety, and $D$ be a $\bQ$-effective $\bQ$-Cartier $\bQ$-divisor on $X$. The Iitaka dimension of $D$, denoted by $\kappa(X,D)$ (or by $\kappa(D)$ if there is no ambiguity), is defined by: 
    $$\kappa(X,D):={\rm sup}\{\dim \Img(\phi_{|mD|})|m\in\bZ_{> 0}\}.$$ Here, $\phi_{|mD|}$ is the rational map $X\dashrightarrow \bP^N$ defined by $x\rightarrow [s_0(x),...,s_N(x)]$, where $\{s_i\}$ is a basis of $H^0(X,\lfloor mD\rfloor)$. If $D$ is not $\bQ$-effective, then $H^0(X,\lfloor mD\rfloor)$ is always empty, so we define $\kappa(D)=-\infty$.
\end{definition}

If $\kappa(X,D)\geq 0$, then there is another equivalent definition of the Iitaka dimension:

\begin{proposition}
    (cf.\cite{U}): $$\kappa(X,D):={\rm sup}\{k\in \bN|\empty \limsup_{m\rightarrow\infty}\frac{\dim H^0(X,\lfloor mD\rfloor)}{m^k}>0\}.$$
\end{proposition}

\begin{definition}
     (Kodaira dimension) Let $X$ be a smooth projective variety. The Kodaira dimension of $X$, denoted $\kappa(X)$, is defined as the Iitaka dimension of $K_X$. If $X$ is not smooth, then we define $\kappa(X)$ be the Kodaira dimension of a resolution of singularity of $X$. This definition is independent to the choice of resolution.
\end{definition}

\begin{definition}
    Let $X$ be a projective variety, and $D$ be a pseudo-effective $\bQ$-Cartier divisor on $X$. The numerical dimension of $D$, denoted by $\nu(X,D)$ (or by $\nu(D)$ if there is no ambiguity), is defined by: $$\nu(X,D):={\rm sup}\{k\in \bN|\empty \limsup_{m\rightarrow\infty}\frac{\dim H^0(X,\lfloor mD+A\rfloor)}{m^k}>0\},$$
    Where $A$ is a fixed ample Cartier divisor on $X$. If $D$ is not pseudoeffective, then we set $\nu(D)=-\infty$. 
\end{definition}

By definition, it is obviously that $\nu(X,D)\geq\kappa(X,D)$ for any $\bQ$-Cartier $\bQ$-divisor $D$. Also, if $D$ is nef, then we have (cf. \cite[Proposition V.2.7(6)]{N}): $$\nu(X,D)= \sup\{k\in n| D^k \not\equiv 0\}.$$

Next, we will introduce the notion of nef reduction and nef dimension, which is proved by Tsuji and Bauer et. al. in \cite{T} and \cite{BCE}.
\begin{theorem}\label{nr}
    Let $X$ be a normal projective variety, and $L$ be a $\bQ$-Cartier divisor on $X$. Then there exists an almost holomorphic dominant rational map $ f : X \rightarrow Y$ with connected fibres, called the nef reduction of L, such that:
    \begin{enumerate}
        \item $L$ is numerically trivial on all compact fibres $F$ of $f$ with $\dim F =\dim X-\dim Y$.
        \item For every very general point $x\in X$ and every irreducible curve $C$ on $X$ passing through $x$ and not contracted by $f$, we have $L.C > 0$. In other words, there exists a subset $Z\subset X$, which is a union of at most countably infinitely many proper closed subsets of $X$, such that if $C$ is a curve on $X$ with $L.C=0$, then either $f(C)$ is a point or $C\subset Z$.
    \end{enumerate}
    The map $f$ is unique up to the birational equivalence of $Y$. The nef dimension of $L$, denoted by $n(X,L)$ (or $n(L)$ if there is no ambiguity), is defined by $n(L):=\dim Y$.
\end{theorem}

Here we recall some important properties of nef dimension.

\begin{proposition}
Let $X$ be a normal projective variety and $L$ be a nef $\bQ$-Cartier divisor on $X$. Then:
    \begin{enumerate}
        \item (cf. \cite[Proposition 2.8]{BCE}) $\kappa(L)\leq\nu(L)\leq n(L)$;
        \item (cf. \cite[Proposition 2.1]{BP}) Let $f: Y\rightarrow X$ be a surjective morphism between normal projective varieties, then we have $n(f^*L)=n(L)$.
    \end{enumerate}
\end{proposition}
In the second case of above, it is easy to check that if $\pi: X\dashrightarrow Z$ is a nef reduction map of $L$, then $\pi\circ f:Y\dashrightarrow Z$ is a nef reduction map of $f^*L$.

\begin{definition}
    Let $X$ be a normal projective variety and $L$ be a pseudo-effective $\bQ$-Cartier divisor on $X$. We say $L$ is  {\em abundant} (or {\em good}) if the equality $\kappa(L)=\nu(L)$ holds. 
\end{definition}
By \cite[Remark 2.3]{A}, if $L$ is nef and abundant, then $\kappa(L)=\nu(L)=n(L)$. Moreover, if $L$ is semiample, then the nef reduction map coincides with the semiample fibration of $L$, and hence $L$ is abundant. In general, studying nef divisors that are not abundant is very difficult. However, for smooth projective surfaces, Ambro provides the following classification of nef divisors and those with maximal nef dimension:
\begin{theorem}(cf. \cite[Theorem 0.3]{A})
    Let $X$ be a smooth projective surface, and $L$ be a nef divisor with $n(L)=2$, then one of the following holds:
    \begin{enumerate}
    \item $D$ itself is big ;
    \item $D$ is not big, but $K_{X}+tD$ is big for all $t>2$;
    \item there is a birational morphism $\mu:X\rightarrow Y$ between smooth projective surfaces such that $sD\equiv \mu^*(-K_{Y})$ for some positive rational number $0<s\leq 2$, and $D$ is algebraically equivalent  to an effective $\bQ$-divisor $D_e$.
    \end{enumerate}
\end{theorem}
This classification is very helpful for studying nef divisors with $n(L)=2$ on higher dimensional varieties by using the nef reduction map of $L$. 

\section{The Generalized Nonvanishing Conjecture for divisors with small nef dimension}

In this section, we will discuss the Generalized Nonvanishing Conjecture for nef divisors with small nef dimensions. To begin, by \cite[Proposition 2.11]{BCE}, a nef divisor $L$ on a smooth projective variety $X$ with $n(L)=1$ is numerically semiample. We aim to extend this result to varieties with rational singularities by applying the birational descent trick for nef divisors in \cite[Section 2 and Section 3]{LP}. 

\begin{lemma}\label{nd1}
Let $X$ be a normal projective variety with rational singularities, and $L$ be a nef $\bQ$-Cartier divisor on $X$. If $n(L)\leq 1$, then $L$ is numerically semiample. 
\end{lemma}
\begin{proof}
By replacing $L$ with $mL$ for some positive integer $m$, we may assume $L$ is Cartier.  Let $\phi:X'\rightarrow X$ be a resolution of singularities of $X$. By \cite[Lemma 2.14]{LP}, if $\phi^*L$ is numerically semiample, then $L$ is also numerically semiample. Since $n(\phi^*L)=n(L)$, it suffices to prove this lemma under the assumption $X$ is smooth. 

Now, if $n(L)=0$ then $L\equiv 0$, and hence the result is trivial. If $n(L)=1$, let $f:X\dashrightarrow Y$ be a nef reduction map of $L$ with general fiber $F$ such that $Y$ is smooth. Then we have $L|_F\equiv 0$. Let $\pi_0:X_0\rightarrow X$ be a resolution of indeterminacy of $f$ with $X_0$ smooth, and let $f_0:X_0\rightarrow Y$ be the induced map. Since the nef reduction is almost holomorphic, $F$ is still a general fiber of $f_0$. Now, for any birational morphism $\mu:Y'\rightarrow Y$ from a smooth projective varieties, there is a diagram
\begin{center}
    \begin{tikzcd}
&X'\arrow[r,"\pi''"]&\tilde{X}\arrow[r,"\pi'"]\arrow[d,"f'"] &X_0\arrow[d,"f_0"]\arrow[r,"\pi_0"]&X\\
& &Y'\arrow[r,"\mu"] &Y &
\end{tikzcd}
\end{center}
such that $\pi''$, $\pi'$ and $\mu$ are birational,  $\tilde{X}$ is the normalization of the main component of $X_0\times_{Y} Y'$, and $X'$ is a resolution of singularities of $\tilde{X}$. By \cite[Lemma 3.1]{LP} for good choice of $Y'$, there is a $\bQ$-divisor $D'$ on $Y'$ such that $f'^*D'\equiv \pi'^*\pi_0^*L$ (in our case, since $\dim Y=1$, $Y'=Y$ and $\mu$ is the identity map). Let $f'':=f'\circ \pi'':X'\rightarrow Y'$ and $\pi'':=\pi''\circ \pi':X'\rightarrow X_0$, we still have $f''^*D'\equiv \pi''^*\pi'^*\pi_0^*L$. Therefore, by \cite[Lemma 2.14]{LP}, to show $L$ is numerically semiample, it suffices to show that $D'$ is numerically semiample. Note that $\dim(Y')=n(L)=1$. Thus, any numerically non-trivial nef divisor on $Y'$ is ample. Hence $D'$ is ample and we are done.
\end{proof}

By \cite[Thoerem 5.22]{KM}, if $(X,\Delta)$ is a klt pair, then $X$ has rational singularities. Therefore, Lemma \ref{nd1} implies $N_{d,k,q,1}$ is true if $k\geq 0$. Employing a similar approach to the above proof, along with Ambro's classification of nef divisors on smooth projective surfaces with maximal nef dimension (\cite[Theorem 0.3]{A}), and Hashizume's theorem for a generalization of the Iitaka Conjecture under the assumption that general fibers have an abundant canonical divisor  (\cite[Theorem 1.4]{H}), we can prove $N_{d,k,q,2}$ is true if $d\leq 5$:

\begin{lemma}\label{nd2}
Let $(X,\Delta)$ be a klt pair of normal projective variety. Let $L$ be a nef $\bQ$-Cartier divisor on $X$ with $n(L)\leq 2$ and $f:X\dashrightarrow Y$ be a nef reduction map of $L$ with general fiber $F$. Assume $\dim F\leq 3$. Let $\Delta_F$ be the $\bQ$-divisor on $F$ defined by $(K_X+\Delta)|_F=K_F+\Delta_F$. Suppose that $\kappa(F,K_F+\Delta_F)\geq 0$, then $K_{X}+\Delta+tL$ is numerically $\bQ$-effective for $t\gg 0$. In particular, if $\kappa(K_X+\Delta)\geq 0$, then for every $t\geq 0$ $K_X+\Delta+tL$ numerically $\bQ$-effective.
\end{lemma}

\begin{proof}

This statement is trivial if $n(L)=0$, so we may assume $n(L)\geq 1$. As in the proof of Lemma \ref{nd1}, we may assume $L$ is Cartier and there is a diagram
\begin{center}
    \begin{tikzcd}
&X'\arrow[r,"\pi"]\arrow[d,"f'"] &X\\
&Y' &
\end{tikzcd}
\end{center}
such that $X'$ is and $Y'$ are smooth, $\pi$ is birational, $f'$ is an algebraic fiber space with general fibers whose general fibers $F'$ is birational to $F$, and there is a $\bQ$-divisor $D'$ on $Y'$ such that $f'^*D'\equiv \pi^*L$. Moreover, by replacing $X'$ with a log resolution of $\Exc(\pi)\cup\pi^{-1}( \Supp(\Delta))$, we may assume $\Supp(\Delta')\cup \Exc(\pi)$ is simple normal crossing, where $\Delta'$ be the proper transform of $\Delta$ on $X'$.

Note that $D'$ is nef with $n(D')=n(L)=n(\pi^*L)\leq 2$. Now, we write $$K_{X'}+\Delta'+E^-=\pi^*(K_X+\Delta)+E^+,$$ where both $E^+,E^-$ are effective $\pi$-exceptional $\bQ$-divisors such that $\Supp E^+$ and $\Supp E^-$ has no common components. Note that $(X',\Delta'+E^-)$ is klt since $\Delta'+E^-$ has snc support and every coefficient 
of $\Delta'+E^-$ is less than 1 (because we had assume $(X,\Delta)$ is klt). Moreover, since $\dim F\leq 3$ and $$(K_{X'}+\Delta'+E^-)|_{F'}=(\pi^*(K_X+\Delta)+E^+)|_{F'}=\pi^*(K_F+\Delta_F)+E^+|_{F'}$$ has non-negative Kodaira dimension, $(K_{X'}+\Delta'+E^-)|_{F'}$ is abundant. Therefore, by \cite[Theorem 1.4]{H}, for any $\bQ$-effective $\bQ$-Cartier divisor $M$ on $Y'$, we have $$\kappa(X', K_{X'}+\Delta'+E^--f'^*K_{Y'}+ f'^*M)\geq \kappa(F',(K_{X'}+\Delta'+E^-)|_{F'})) +\kappa(Y',M)\geq\kappa(Y',M).$$

Now, we consider the question case by case on $n(L)=\dim Y'$. In the case $n(L)=1$, $Y'$ is a curve, and hence $D'$ is ample. This shows that for $t\gg 0$ $$\kappa(X', K_{X'}+\Delta'+E^--f'^*K_{Y'}+ f'^*(K_{Y'}+tD'))\geq \kappa(Y',K_{Y'}+tD')= 1.$$ Therefore, $K_{X'}+\Delta'+E^-+t\pi^*L$ is numerically equivalent to the $\bQ$-effective divisor $K_{X'}+\Delta'+E^--f'^*K_{Y'}+ f'^*(K_{Y'}+tD')$. Note that we have $$K_{X'}+\Delta'+E^-+t\pi^*L+E_1=\pi^*(K_X+\Delta+tL)+E_2$$ for some $\pi$-exceptional effective $\bQ$-divisors $E_1$ and $E_2$. This implies $\pi^*(K_X+\Delta+tL)+E_2+L_0'$ is effective for some numerically trivial $\bQ$-Cartier $\bQ$-divisor $L_0'$. By \cite[Theorem 5.22]{KM}, both $X$ and $X'$ has rational singularities, hence there exists a numerically trivial $\bQ$-Cartier $\bQ$-divisor $L_0$ on $X$ such that $\pi^*L_0=L_0'$. This means $\pi^*(K_X+\Delta+tL+L_0)+E_2$ is effective, which implies $K_X+\Delta+tL+L_0$ is effective, hence $K_X+\Delta+tL\equiv K_X+\Delta+tL+L_0$ is numerically $\bQ$-effective.

In the case $n(L)=2$, since $Y'$ is smooth and $\dim Y'=n(D')=2$, by \cite[Theorem 0.3]{A}, there are three possible cases:
\begin{enumerate}
    \item $D'$ itself is big (hence $K_{Y'}+tD'$ is big for $t\gg 0$);
    \item $D'$ is not big, but $K_{Y'}+tD'$ is big for all $t>2$;
    \item there is a birational morphism $\mu:Y'\rightarrow Y_0$ between smooth projective surfaces such that $sD'\equiv \mu^*(-K_{Y_0})$ for some positive rational number $s$, and $D'$ is algebraically equivalent (hence numerically equivalent) to an effective $\bQ$-divisor $D_e$.
\end{enumerate}
Now, in the first two cases, $K_{Y'}+tD'$ is big for $t\gg 0$, hence we have $$\kappa(X', K_{X'}+\Delta'+E^--f'^*K_{Y'}+ f'^*(K_{Y'}+tD'))\geq\kappa(Y',K_{Y'}+tD')= 2.$$ Therefore, $K_{X'}+\Delta'+E^-+t\pi^*L$ is numerically equivalent to the $\bQ$-effective divisor $K_{X'}+\Delta'+E^--f'^*K_{Y'}+ f'^*(K_{Y'}+tD')=K_{X'}+\Delta'+E^- +tf'^*D'$. Hence by the same argument above, $K_X+\Delta+tL$ is numerically $\bQ$-effective. 

In the last case, we have $K_{Y'}+(-\pi^*K_{Y_0})=G$ is effective because $Y_0$ is smooth. Moreover, we have $$G+(t-s)D_e=K_{Y'}+(-\pi^*K_{Y_0})+(t-s)D_e\equiv K_{Y'}+tD'.$$ Note that $M:=G+(t-s)D_e=K_{Y'}$ is effective if $t\gg 0$. Thus, by \cite[Theorem 1.4]{H} again, we have
$$\kappa(X', K_{X'}+\Delta'+E^--f'^*K_{Y'}+ f'^*M)\geq\kappa(Y',M)\geq 0.$$ Since $f'^*M\equiv f'^*(K_{Y'}+tD')\equiv f'^*K_{Y'}+t\pi^*L$, $K_{X'}+\Delta'+E^-+t\pi^*L$ is numerically equivalent to the effective $\bQ$-divisor $K_{X'}+\Delta'+E^--f'^*K_{Y'}+ f'^*M$, hence by the same argument of above, $K_X+\Delta+tL$ is numerically $\bQ$-effective.

For the last statement, if $t=0$ then the statement is trivial. If $t>0$, then pick $s\gg t>0$, we have $$K_X+\Delta+tL = \frac{t}{s}((\frac{s}{t}-1)(K_X+\Delta)+K_X+\Delta+sL).$$ Since $K_X+\Delta$ is $\bQ$-effective and $K_X+\Delta+sL$ is numerically $\bQ$-effective for $t\gg 0$, hence we are done.
\end{proof}

The following corollary is directly follows from Lemma \ref{nd2}.
\begin{corollary}\label{nd2c}
Let $(X,\Delta)$ be a klt pair of normal projective variety such that $K_X+\Delta$ is pseudoeffective, and $L$ be a nef Cartier divisor on $X$.
\begin{enumerate}
    \item If $\dim =3$ and $0\leq n(L)\leq 2$, then for every $t\geq 0$, $K_X+\Delta+L$ is numerically $\bQ$-effective;
    \item If $\dim =4$  and $1\leq n(L)\leq 2$, then for $t\gg 0$, $K_X+\Delta+tL$ is numerically $\bQ$-effective;
    \item If $\dim =4$, $1\leq n(L)\leq 2$, and $\kappa(K_X+\Delta)\geq 0$, then for every $t\geq 0$, $K_X+\Delta+tL$ is numerically $\bQ$-effective;
    \item If $\dim =5$ and $n(L)=2$, then for $t\gg 0$, $K_X+\Delta+tL$ is numerically $\bQ$-effective;
    \item If $\dim =5$,$n(L)=2$, and $\kappa(K_X+\Delta)\geq 0$, then for every $t\geq 0$, $K_X+\Delta+tL$ is numerically $\bQ$-effective.
\end{enumerate}
\end{corollary}

\begin{remark}
    Our proof of Lemma \ref{nd2} differs from the proof in \cite{C23}, although both rely on the nef reduction. The difference is as the following: the proof in \cite{C23} use the nef reduction of $K_X+B+L$, which implies $(K_X+B)|_F\equiv 0 \equiv L|_F$ if $F$ is a general fibre of the nef reduction. Hence it can be shown that a good birational modification of the nef reduction concerning $K_X+B+L$ is a klt-trivial fibration. Therefore, by employing the theory of klt-trivial fibrations and Zariski decomposition, one can deduce $K_X+B+L$ is numerically semiample. 
    
    However, in our proof $K_X+B+L$ may not be nef, and we are using the nef reduction of $L$, which means we cannot control the restriction of $K_X+B$ on the general fibre $F$. Also, although by \cite[Theorem 4.1]{LP}, one can take a birational modification $\pi:X\dashrightarrow X'$ to make $K_{X'}+\pi_*\Delta+t\pi_*L$ be nef for $t\gg 0$, in this case, we cannot control the nef dimension of $K_{X'}+\pi_*\Delta+t\pi_*L$. 
\end{remark}

\section{The Generalized Nonvanishing Conjecture for nef divisors with maximal nef dimension via the Albanese map}

In this section, we will prove the Generalized Nonvanishing Conjecture holds for nef divisor with maximal nef dimension if the relative dimension of the Albanese map is at most 2 over its image. This holds, in particular, when $\dim X=3$ with $q>0$. Before we start the proof, we give several lemmas proved by standard methods in the minimal model program.

\begin{lemma}\label{mor}
Let $(X,\Delta)$ be a klt pair of normal projective variety, $f:X'\rightarrow X$ be a log resolution of $(X,\Delta)$, and $\Delta'$ be the proper transform of $\Delta$ on $X'$. Write $$K_{X'}+\Delta'+E^-=f^*(K_X+\Delta)+E^+$$ for some $f$-exceptional effective $\bQ$-divisors $E^+,E^-$ such that $\Supp E^+$ and $\Supp E^-$ has no common components. Let $L$ be a $\bQ$-Cartier divisor on $X$. Suppose $K_{X'}+\Delta'+E^-+f^*L$ is $\bQ$-effective (resp. numerically $\bQ$-effective), then $K_X+\Delta+L$ is $\bQ$-effective (resp. numerically $\bQ$-effective). In particular, since $(X',\Delta'+E^-)$ is a smooth klt pair with $\kappa(K_{X'}+\Delta'+E^-)=\kappa(K_X+\Delta)$, to prove $N_{d,k,q,n}$, it suffices to show $N_{d,k,q,n}$ holds for smooth pair $(X,\Delta)$.
\end{lemma}
\begin{proof}
We have $$ K_{X'}+\Delta'+E^-+f^*L= f^*(K_X+\Delta+L)+E^+. $$ If $K_{X'}+\Delta'+E^-+f^*L= f^*(K_X+\Delta+L)+E^+$ is $\bQ$-effective, then $K_X+\Delta+L$ is also $\bQ$-effective since $E^+$ is $f$-exceptional. If $K_{X'}+\Delta'+E^-+f^*L$ is numerically $\bQ$-effective, there exists a numerically trivial $\bQ$-divisor $L_0'$ on $X'$ such that $f^*(K_X+\Delta+L)+E^++L_0'$ is an effective $\bQ$-Cartier divisor. Again, since both $X$ and $X'$ have rational singularities, there exists a numerically trivial $\bQ$-divisor on $L_0$ on $X$ such that $f^*L_0=L_0'$. Thus,  $f^*(K_X+\Delta+L+L_0)+E^+=f^*(K_X+\Delta+L)+E^++L_0'$ has non-negative Kodaira dimension, and so does $K_X+\Delta+L+L_0$. Hence $K_X+\Delta+L\equiv K_X+\Delta+L+L_0$ is numerically $\bQ$-effective.
\end{proof}


\begin{lemma}\label{bigc}
    Let $(X,\Delta)$ be a $\bQ$-factorial klt pair, $R$ be a $(K_X+\Delta)$-negative extremal ray of $X$, such that the contraction for $R$ is not of fiber type. Assume $(K_X+\Delta)$-flip exists. Let $\pi:X\dashrightarrow Y$ be the corresponding divisorial contraction or $(K_X+\Delta)$-flip. Let $L$ be a $\bQ$-Cartier $\bQ$-divisor on $X$ such that $L.R\leq 0$, and $t$ is a non-negative rational number, then we have $$\kappa(K_X+\Delta+tL)=\kappa(K_Y+\pi_*\Delta+t\pi_*L).$$ 
    In particular, if $K_Y+\pi_*\Delta+t\pi_*L$ is big, then so does $K_X+\Delta+tL$.
\end{lemma}
\begin{proof}
    Note that $(Y,\pi_*\Delta)$ is also klt $\bQ$-factorial. If $\pi$ is a flip then $X$ and $Y$ are isomorphic in codimension 2, hence the equality is trivial. Therefore, we may assume $\pi$ is a divisorial contraction. In this case, we can write $\pi^*\pi_*L+E_L= L$ for some $\pi$-excpetional $\bQ$-divisor $E_L$. Since $L.R\leq 0$, $E_L$ is effective by the negativity lemma. Let $\Delta=\Delta_1+\Delta_2$ such that both $\Delta_1, \Delta_2$ are effective, $\Supp\Delta_1$ and $\Supp\Delta_2$ have no common component, every component of $\Delta_1$ is not contracted by $\pi$, and $\Delta_2$ is $\pi$-exceptional. Then $\Delta_1$ is the proper transform of $\pi_*\Delta$ on $X$, and we have 
    \begin{align*}
        K_X+\Delta+tL &= K_X+\Delta_1+tL+\Delta_2\\
        &= \pi^*(K_X+\pi_*\Delta)+E+tL+\Delta_2\\
        &= \pi^*(K_Y+\pi_*\Delta+t\pi_*L) + (E+tE_L+\Delta_2)
    \end{align*}
    for some $\pi$-exceptional $\bQ$-divisor $E$. Note that $E$ itself may not be effective, but we have $(E+tE_L+\Delta_2).R=(K_X+\Delta+tL).R<0.$ Since $\Delta_2$ is also $\pi$-exceptional, by the negativity lemma again we conclude that $tE_L+E+\Delta_2$ is effective. Hence the equality $$\kappa(K_X+\Delta+tL)=\kappa(K_Y+\pi_*\Delta+t\pi_*L)$$ holds.
\end{proof}

\begin{corollary}\label{sfb}
    Let $(X,\Delta)$ be a 2-dimensional klt pair such that $\kappa(K_X+\Delta)\geq 0$, and $L$ be a nef $\bQ$-Cartier $\bQ$-divisor on $X$ with $n(L)=2$, then $K_X+\Delta+tL$ is big for $t\gg 0$.
\end{corollary}
\begin{proof}
    First, we consider the special case that there is no $(K_X+\Delta)$-negative extremal face $R$ of $X$ with $L.R=0$. In this case, we have $K_X+\Delta+tL$ is nef for $t\gg 0$ with $n(K_X+\Delta+tL)=2$, so by \cite[Corollary C]{LP}, $K_X+\Delta+tL \equiv M$ for some semiample $\bQ$-divisor $M$ with $n(M)=2$. Since the nef dimension of a semiample $\bQ$-divisor is equivalent to the Kodaira dimension, which implies $M$ is big, hence so does $K_X+\Delta+tL$.
    
    In general, we will construct a $L$-trivial $(K_X+\Delta)$-MMP $X\rightarrow Y$. More precisely, let $R$ be a $(K_X+\Delta)$-negative extremal ray with $L.R=0$, and $\pi_1:X\rightarrow X_1$ be the corresponding contraction. Note that $\pi_1^*(\pi_1)_*L\equiv L$ and hence $(\pi_1)_*L$ is nef with $n((\pi_1)_*L)=n(L)=2$ (here $(\pi_1)_*L$ is $\bQ$-Cartier since 2-dimensional klt pair is always $\bQ$-factorial). If there is no $(K_{X_1}+(\pi_1)_*\Delta)$-negative extremal ray $R_1$ with $(\pi_1)_*L.R_1=0$, then we let $Y=X_1$, otherwise, we let $R_1$ be a such extremal ray and $\pi_2: X_1 \rightarrow X_2$ be the corresponding contraction. Then we consider the same process on $X_2$ again. Inductively, we obtained a sequence $$X=X_0\xrightarrow{\pi_1} X_1 \xrightarrow{\pi_2} X_2\xrightarrow{\pi_3}...\xrightarrow{\pi_n} X_n=Y$$ such that there is no $(K_{X_n}+(\phi_n)_*\Delta)$-negative extremal ray $R_n$ with $(\phi_n)_*L.R_n=0$ (here we define $\phi_k:=\pi_k\circ\pi_{k-1}\circ...\circ\pi_1: X\rightarrow X_k$), which implies $K_{X_n}+(\phi_n)_*\Delta+t(\phi_n)_*L$ is big. By Lemma \ref{bigc}, we conclude $K_{X_{n-1}}+(\phi_{n-1})_*\Delta+t(\phi_{n-1})_*L$ is also big. Inductively, since each $\pi_i$ is a $(K_{X_{i-1}}+(\phi_{i-1})_*\Delta)$-negative extremal contraction such that $(\phi_{i-1})_*L$ is $\pi_i$-trivial, by applying Lemma \ref{bigc} on each $\pi_i$ inductively, we conclude that $K_X+\Delta+tL$ is big.
\end{proof}

Now we can prove the following proposition:

\begin{proposition}\label{alb}
Let $(X,\Delta)$ be klt pair of a normal projective variety such that $K_X+\Delta$ is pseudoeffecitve and $q(X)>0$. Let $\alpha: X\rightarrow A(X)$ be the Albanese map and $F$ be a general fibre of $\alpha$. If $\dim F\leq 2$, then for any nef $\bQ$-Cartier divisor $L$ with maximal nef dimension, there is a numerically trivial $\bQ$-Cartier divisor $D$ on $A$ such that $\kappa(K_X+\Delta+tL+\alpha^*D)\geq 0$ for $t\gg 0$. In particular, if $\kappa(K_X+\Delta)\geq 0$, then for every $t\geq 0$, $K_X+\Delta+tL$ is numerically $\bQ$-effective.
\end{proposition} 
\begin{proof}
Let $f:X'\rightarrow (X,\Delta)$ be a log resolution of $(X,\Delta)$, and write $K_{X'}+\Delta'+E^-=f^*(K_X+\Delta)+E^+$ as before (here $\Delta'$ is the proper transform of $\Delta$ on $X'$). Let $\alpha':X'\rightarrow A(X')$ be the Albanese map of $X'$. Since $X$ has rational singularities, by \cite[Lemma 2.4.1]{BS} we have $A(X')=A(X)$ and $\alpha'=\alpha\circ f$. Assume our result holds for $(X', \Delta'+E^-)$ and $f^*L$, that is, $\kappa(K_{X'}+\Delta'+E^-+tf^*L+\alpha'^*D)\geq 0$ for $t\gg 0$. Since $K_{X'}+\Delta'+E^-+tf^*L+\alpha'^*D=K_{X'}+\Delta'+E^-+f^*(tL+\alpha^*D)$, then by Lemma \ref{mor}, $K_X+\Delta+tL+\alpha^*D$ also has non-negative Kodaira dimension. Therefore, we can replace $(X,\Delta)$ and $L$ by $(X', \Delta'+E^-)$ and $f^*L$, hence we may assume $(X,\Delta)$ is smooth. 

Let $F$ be a general non-empty fiber of $\alpha$. By the generic smoothness, we may assume every connected component of $F$ is smooth and irreducible. Also, by \cite[Theorem 4.1]{BC}, it suffices to show that $K_F+\Delta_F+tL_F$ is relatively big for $t\gg 0$, where $\Delta_F$ is defined by $(K_X+\Delta)|_F=K_F+\Delta_F$, and $L_F:=L|_F$. Note that if $\kappa(K_X+\Delta)\geq 1$, then $\alpha$ may not be surjective and possibly has disconnected fibers, but \cite[Theorem 4.1]{BC} does not require the morphism to be surjective or has connected fibers.

Since we assume $L$ is of maximal nef dimension, $L_F$ is also of maximal nef dimension on $F$, hence there is at least an irreducible component of $F$, say $Z$, satisfying that $\dim Z=\dim F$ and $L_Z:= L|_Z$ is still of maximal nef dimension. Also, we have $K_F|_Z=K_Z$ since $Z$ is a smooth connected component of $F$. Defining $\Delta_Z:=\Delta_F|_Z$, then we have $K_Z+\Delta_Z+tL_Z=(K_F+\Delta_F+tL_F)|_Z$.

Now, if $\dim Z=0$, then $\alpha$ is a generically finite morphism over its image, and hence any divisors are always relatively big. If $\dim Z=1$, then $Z$ is a smooth projective curve. Therefore, we have $K_Z+tL_Z$, and hence $K_Z+\Delta_Z+tL_Z$ is ample for $t\gg 0$, which implies $K_F+\Delta_F+tL_F$ is big. If $\dim Z=2$, then $Z$ is a smooth projective surface with $\kappa(K_Z+\Delta_Z)\geq 0$ since $K_X+\Delta$ is pseudoeffecitve. Thus, by Corollary \ref{sfb}, $K_Z+\Delta_Z+tL_Z$ is big for $t\gg 0$, which also implies the bigness of $K_X+\Delta_F+tL_F$. Hence we proved the bigness of $K_F+\Delta_F+tL_F$.
\end{proof}

\begin{corollary}\label{hps}
    Let $(X,\Delta)$ be a klt pair of normal projective variety such that $K_X+\Delta$ is pseudo-effective, and $L$ be a nef divisor on $X$ such that $n(L)=\dim X$. Assume one of the following conditions holds: 
    \begin{enumerate}
        \item $\dim X\leq 3$ and $q(X)>0$;
        \item $\kappa(K_X+\Delta)=0$ and $\dim X- q(X)\leq 2$.
    \end{enumerate}
    Then for every $t\geq 0$, $K_X+\Delta+tL$ is numerically $\bQ$-effective:
\end{corollary}
\begin{proof}
    The first case immediately follows from Proposition \ref{alb} by the usual Nonvanishing Theorem in dimension 3. For the second case, we only need to show the Albanese map is surjective (If $X$ is smooth and $\Delta=0$, it is well-known that $\alpha$ is an algebraic fibre space by \cite{K81}). Let $\alpha:X\rightarrow A(X)$ be the Albanese morphism, $A_0$ be the normalization of $\Img\alpha$, and $X\rightarrow A'$ be the Stein factorization of $X\rightarrow A_0$ (this is a morphism since $X$ is normal). By \cite[Theorem 4.22]{CP} and the last remark in \cite[Section 1]{HPS}, we can prove the inequality in \cite[Theorem 1.1]{HPS} holds when $X$ replaced by a klt pair $(X,\Delta)$ (see Remark \ref{hps2}). This implies $\kappa(A')\leq 0$. On the other hand, $A'$ has maximal Albanese dimension, which implies $\kappa(A')\geq 0$. Hence $\kappa(A')=0$ and the image of $A'$ in $A$, which is just $\Img\alpha$, is an abelian subvariety of $A$. By the universal property of Albanese, we conclude that $A'=\Img\alpha=A(X)$ and hence $\alpha$ is an algebraic fibre space over $A(X)$, in particular, $\alpha$ is surjective.
\end{proof}

\begin{remark}\label{hps2}
As mentioned in \cite{HPS}, by using \cite[Theorem 4.22]{CP}, it is easy to show the inequality in \cite[Theorem 1.1]{HPS} holds for klt pairs by making a very little modification of the proof of the last part of \cite[Lemma 5.1]{HPS}. For the convenience of readers, we provide the details here: Let $(X,\Delta)\rightarrow Y$ be an algebraic fibre space between normal projective varieties with general fibre $F$ such that $(X,\Delta)$ is a klt pair of with $\kappa(K_X+\Delta)=0$ and $Y$ has maximal Albanese dimension. Let $\pi:X'\rightarrow X$ be a log resolution of $(X,\Delta)$, write $K_{X'}+\Delta'=\pi^*(K_X+\Delta)+E^+-E^-$ (where $\Delta'$ is the proper transform of $\Delta$ on $X'$). Replace $(X,\Delta)$ with $(X',\Delta'+E^-)$, and replace $F$ by $F'$ (here $F'$ is the general fibre of $X'\rightarrow Y$), we have $(X',\Delta'+E^-)$ is klt, the equality $$\kappa(X',K_{X'}+\Delta'+E^-)=\kappa (X', \pi^*(K_X+\Delta)+E^+)= \kappa(X,K_X+\Delta)$$ holds, and the inequality 
\begin{align*}
    \kappa(F', (K_{X'}+\Delta'+E^-)|_{F'})&= \kappa(F', (\pi^*(K_X+\Delta)+E^+)|_{F'})\\
    &\geq \kappa(F', (\pi^*(K_X+\Delta))|_{F'})= \kappa(F, (K_X+\Delta)|_{F})
\end{align*}
also holds. Thus, if the inequality in \cite[Theorem 1.1]{HPS} holds for $(X',\Delta'+E^-)\rightarrow Y$, then so does $(X,\Delta)\rightarrow Y$, hence we may assume the pair $(X,\Delta)$ is smooth. 

Let $Y_0:=\Img (Y\rightarrow A(Y))$ and $Y\rightarrow Y'$ be the Stein factorization of $Y\rightarrow Y_0$. Since $Y\rightarrow Y_0$ is generically finite, we have $Y\rightarrow Y'$ is birational,  $Y'$ is normal, $X\rightarrow Y'$ is an algebraic fibre space, and $Y'\rightarrow A(Y)$ is finite over its image (which is just $Y_0$). By \cite[Theorem 13]{K81}, there is an \'etale covering $Y''\rightarrow Y'$ such that $Y''\cong Z\times K$, where $Z$ is of general type with $\kappa(Z)=\kappa(Y')=\kappa(Y)$ and $K$ is an abelian variety. Let $X'':= X\times_{Y'} Y''$, since $Y''\rightarrow Y'$ is \'etale, so does the induced may $g:X''\rightarrow X$. Hence $X''$ is smooth since $X$ is. Let $\Delta'':=g^*(\Delta)$, then we have $K_{X''}+\Delta''=g^*(K_X+\Delta)$ and $(X'',\Delta'')$ is klt since $g$ is \'etale. Consider the morphism $h:(X'',\Delta'')\rightarrow Z$ induced by $X''\rightarrow Y''=K\times Z$, and let $G$ be a general fibre of $h$. Then we have the induced morphism $G\rightarrow K$, which has general fibre $F$. Since $(G, \Delta''|_G)$ is klt, by \cite[Theorem 4.22]{CP}, we have $$\kappa(G, (K_{X''}+\Delta'')|_G)=\kappa(G, K_G+\Delta''|_G)\geq \kappa(F,(K_{X''}+\Delta'')|_F)= \kappa(F,(K_{X}+\Delta)|_F).$$ Since $Z$ is of general type, by \cite[Theorem 1.7]{F17} (or \cite[Section 4]{Ca}), we have $$\kappa(X'', K_{X''}+\Delta'')\geq\kappa(G, (K_{X''}+\Delta'')|_G)+\dim Z \geq\kappa(F,(K_{X}+\Delta)|_F)+\dim Z. $$
Since $\kappa(X'', K_{X''}+\Delta'')= \kappa (X, K_{X}+\Delta)$ and $\dim Z=\kappa (Y')=\kappa (Y)$, the proof is completed.
\end{remark}

\begin{proof}[Proof of Theorem \ref{nvc}]
The four cases in Theorem \ref{nvc} immediately follow from Corollary \ref{nd2c} and Corollary \ref{hps}.
\end{proof}

\section{Application on Iitaka Conjecture}

At first, we prove Theorem \ref{basenvt}. The argument of our proof almost follows the idea of Birkar's proof in \cite[Theorem 1.3]{B}, with a little modification by using the Generalized Nonvanishing Conjecture.

\begin{proof}[Proof of Theorem \ref{basenvt}]
Since this statement is trivial if $\kappa(Y)=-\infty$, we may assume $\kappa(Y)\geq 0$. By \cite[Theorem 4.5]{FM}, there is a commutative diagram 
\begin{center}
    \begin{tikzcd}
&X'\arrow[r,"\pi"]\arrow[d,"f'"] &X\arrow[d,"f"]\\
&Y'\arrow[r,"\mu"] &Y
\end{tikzcd}
\end{center}
such that $X'$ and $Y'$ are smooth, $\pi$ and $\mu$ are birational morphism, and $f'$ is an algebraic fibre space with SNC branch. Let $p>1$ be a positive integer such that $f'_*\cO_{X'}(pK_{X'})\neq 0,$ then there exists an effective $\bQ$-divisor $B$ and a nef $\bQ$-divisor $L$ on $Y'$ satisfies the followings:
\begin{enumerate}
    \item $pK_{X'}=pf'^*(K_{Y'}+B+L)+R$ for some $\bQ$-divisor $R$ on $X'$.
    \item If we write $R=R^+-R^-$ with $R^+,R^-$ both effective and without common components, then $R^-$ is both $\pi$-exceptional and $f'$-exceptional. 
\end{enumerate}
Since we assume $N_{d,k,q}$ holds for $d=\dim Y=\dim Y'$, $k=\kappa(Y)=\kappa(Y')$, and $q=q(Y)=q(Y')$, for any non-negative rational number $j$, $K_{Y'}+jL$ lies in the effective cone ${\rm Eff}(Y')$. In particular, since $B$ is effective, for any $j\geq 0$, there is an effective $\bQ$-divisor $D_j$ such that $D_j\equiv  K_{Y'}+B+jL=:L_j.$

Let $i,j\gg 0$ be sufficiently divisible integers, and $M:= \pi_*f'^*(D_j-L_j)\equiv 0$. Since $R^-$ is $\pi$-exceptional, there exist some effective $\pi$-exceptional divisors $E^+,E^-$ such that $$i\pi^*(jpK_{X}+pM)+E^+-E^-=i(jpK_{X'}+jR^-+pf'^*(D_j-L_j)).$$ Thus, we have
\begin{align*}
    H^0(i(jpK_X+pM))&= H^0(i\pi^*(jpK_X+pM))\\ 
    &= H^0(i\pi^*(jpK_X+pM)+E^+)\\
    &\geq H^0(i\pi^*(jpK_X+pM)+E^+-E^-)\\ 
    &=H^0(i(jpK_{X'}+jR^-+pf'^*(D_j-L_j)))
\end{align*}
Moreover, we have
\begin{align*}
    jpK_{X'}+jR^-+pf'^*(D_j-L_j)&=jpf'^*(K_{Y'}+B+L)+jR^++pf'^*(D_j-L_j)\\
    &=jpf'^*(K_{Y'}+B+L)+jR^++pf'^*(D_j-K_{Y'}-B-jL)\\
    &=(j-1)pf'^*(K_{Y'}+B)+jR^++pf'^*(D_j)\\
    &\geq (j-1)pf'^*(K_{Y'}).
\end{align*}
Here, the first equality follows from the canonical bundle formula of Fujino-Mori, the second equality follows from the definition of $L_j$, and the last inequality holds because all of $B, R^-$, and $D_j$ are effective. By this inequality, for any sufficiently divisible positive integer $i$, we have $$H^0(i(jpK_{X'}+jR^-+pf'^*(D_j-L_j)))\geq H^0((j-1)pf'^*(K_{Y'})) =H^0((j-1)pK_{Y'}).$$ By \cite[Theorem 3.1]{CPT}, we have $$H^0(ijpK_X)\geq H^0(ijpK_X+ipM)\geq H^0(i(jpK_{X'}+jR^-+pf'^*(D_j-L_j))) \geq H^0(i(j-1)pK_{Y'}).$$ This implies $$\kappa(X)=\kappa(X,jpK_X)\geq \kappa(Y,(j-1)pK_Y)=\kappa(Y),$$ which completes the proof.
\end{proof}

\begin{remark}
The above proof also shows that even if we only have $K_Y+B+jL$ is numerically $\bQ$-effective for some fixed $j\in \bQ_{>1}$, the result of the theorem still holds.
\end{remark}

Next, before we prove Theorem \ref{iitaka7}, we give the definition of the augmented irregularity:
\begin{definition}\label{aug}
    Let $X$ be a normal projective variety, the augmented irregularity of $X$, denoted by $\tilde{q}(X)$, is defined by 
    \begin{center}
    $\tilde{q}(X):={\rm max}\lbrace q(X')|\ X'\rightarrow X $  is quasi \'etale $\rbrace$.
    \end{center}
\end{definition}

Now, we have prepared all the ingredients for proving Theorem \ref{iitaka7}. 

\begin{proof}[Proof of theorem \ref{iitaka7}]
We may assume $\kappa(F)\geq 0$ and $\kappa(Y)\geq 0$ (otherwise, this statement is trivial). By \cite{C} and \cite{K82} (or \cite{CP} and \cite{HPS}), $C_{n,m}$ holds if $m\leq 2$. By \cite{K85}, $C_{n,m}$ holds if $F$ has a good minimal model. In particular, by \cite[Theorems 4.4 and 4.5]{L}, \cite[Theorem 1]{K81}, and the existence of good minimal model in dimension 3, $C_{n,m}$ holds if either $\dim F-\kappa(F)\leq 3$, or in the case $\kappa(F)=0$ and $\dim F- q(F)\leq 3$. This implies $C_{n,m}$ holds if either $n-m\leq 3$, or $n-m=4$ but eithre $\kappa(F)\geq 1$ or $\kappa(F)=0$ with $q(F)\geq 1$. Therefore, when $n\leq 7$, $C_{n,m}$ is confirmed except for the case $n=7, m=3,$ and $\kappa(F)=q(F)=0$.

Now, by Corollary \ref{nvc3d} and Theorem \ref{basenvt}, if either $\kappa(Y)>0$ or $q(Y)>0$, then we have $\kappa(X)\geq \kappa(Y)$, hence we are done. If $\kappa(Y)=q(Y)=0$ but $\tilde{q}(Y)>0$, then there exists a quasi-\'etale covering $Y'\rightarrow Y$ (hence \'etale by the Nagata-Zariski purity theorem and the smoothness of $Y$) with $q(Y')=\tilde{q}(Y)>0.$ Let $X':= X\times_{Y} Y'$, then $X'\rightarrow X$ is also \'etale (hence $X'$ is smooth), and the morphism $X'\rightarrow Y'$ induced by the base change is an algebraic fiber space between smooth projective varieties with general fiber $F$. Since the Kodaira dimension is invariant under \'etale covering, by Theorem \ref{basenvt} again, we have the inequality $$\kappa(X)=\kappa(X')\geq \kappa(F)+\kappa(Y')=\kappa(F)+\kappa(Y).$$ Finally, if $q(F)=q(Y)=\tilde{q}(Y)=0$, then by \cite[Theorem 1.6]{F05}, $q(X)=0$. This completes the proof.
\end{proof}

\begin{remark}\label{aug2}
    If $N_{d,k,q,n}$ is true, then we can show $K_X+\Delta+tL$ is numerically $\bQ$-effective for klt pair $(X,\Delta)$ with $\dim X=d$, $\kappa(K_X+\Delta)=k$, $n(L)=n$, and $\tilde{q}(X)=q$ in the following way: By Lemma \ref{mor}, we may assume $(X,\Delta)$ is smooth. Let $\pi:(X',\Delta')\rightarrow (X, \Delta)$ be a quasi-\'etale covering (which is \'etale by the Nagata-Zariski purity theorem and the smoothness of $X$) with $q(X')=\tilde{q}(X)$. Then $(X',\Delta')$ is klt with $K_{X'}+\Delta'=\pi^*(K_X+\Delta)$, hence $\kappa(K_{X'}+\Delta')=\kappa(K_X+\Delta)$. By our assumption, $K_{X'}+\Delta'+t\pi^*L=\pi^*(K_X+\Delta+tL)$ is numerically equivalent to an effective divisor $D$. Now, follows the argument in \cite[Corllary 4.2]{WZ}, since $K_X+\Delta+tL=\frac{1}{\deg \pi}\pi_*(K_{X'}+\Delta'+t\pi^*L)$ is weak numerically equivalent $\frac{1}{\deg \pi}\pi_*D$, by the smoothness of $X$, they are numerically equivalent. Hence we have $$K_X+\Delta+tL=\frac{1}{\deg \pi}\pi_*(K_{X'}+\Delta'+t\pi^*L)\equiv \frac{1}{\deg \pi}\pi_*D\geq 0.$$
\end{remark}

\begin{remark}
    For the case $\kappa(Y)=q(Y)=0$ but $\tilde{q}(Y)>0$ in the above proof, since for smooth projective varieties, the augmented irregularity is also a birational invariant, one can use Remark \ref{aug2} to show the Generalize Nonvanishing Conjecture holds for threefold with $\tilde{q}>0$, and apply Theorem \ref{basenvt} to get our result.
\end{remark}

\begin{remark}
The remaining unknown case of $C_{7,m}$ satisfies the following conditions:
\begin{enumerate}
    \item $m=3$;
    \item $\kappa(F)=\kappa(Y)=q(X)=q(F)=\tilde{q}(Y)=0$;
    \item In the decomposition $pK_{X'}=pf'^*(K_{Y'}+B+L)+R$ in the Fujino-Mori's canonical bundle formula, $\kappa(Y', K_{Y'}+B)=\nu(Y', K_{Y'}+B)= 0$ and $n(L)= 3$.
\end{enumerate}
If one can prove $\kappa(X')\geq 0$ under the above assumptions, then the proof of $C_{7,m}$ is completed.
\end{remark}


\begin{thebibliography}{}
\bibitem[A04]{A}  F. Ambro, {\em Nef dimension of minimal models}, Math. Ann. 330, no {\bf 2} (2004) 309-322.

\bibitem[B09]{B} C. Birkar, {\em The Iitaka conjecture Cn,m in dimension six}, Compositio Math. (2009), {\bf 145}, 1442-1446.

\bibitem[BC15]{BC} C. Birkar, J.A. Chen, {\em Varieties fibered over abelian varieties with fibers of log general type}, Adv. Math. {\bf 270}, (2015), 206-222.

\bibitem[BCE+02]{BCE} T.Bauer, F. Campana, T. Eckl, S. Kebekus, T. Peternell, S. Rams, T. Szemberg, L. Wotzlaw; {\em A reduction map for nef line bundles}, Complex geometry: collection of papers dedicated to Hans Grauert (Gottingen, 2000), 27-36, Springer, Berlin, 2002.

\bibitem[BP04]{BP} T. Bauer and T. Peternell, {\em Nef reduction and anticanonical bundles,} Asian J. Math. {\bf 8} (2004), no. 2, 315–352

\bibitem[BS95]{BS} M. Beltrametti, A. Sommese, {\em The adjunction theory of complex projective varieties}, de Gruyter Expositions in Mathematics, {\bf 16}. Walter de Gruyter $\And$ Co., Berlin, 1995.

\bibitem[C04]{Ca} F. Campana, {\em Orbifolds, special varieties, and classification theory,} Ann. Inst. Fourier (Grenoble), {\bf 54} (2004), no. 3, 499–630.

\bibitem[C18]{C} J. Cao, {\em Kodaira dimension of algebraic fiber spaces over surfaces}, Algebr. Geom. {\bf 5} (2018), no. 6, 728–741.

\bibitem[C23]{C23} P. Chaudhuri, {\em An inductive approach to generalized abundance using nef reduction}, Comptes Rendus Math. 2023, Vol. {\bf 361}, p. 417-421.

\bibitem[CP17]{CP} J. Cao, M. P\v aun,{\em Kodaira dimension of algebraic fiber spaces over Abelian varieties},  Invent. Math. {\bf 207} (2017), no. 1, pp 345-387

\bibitem[CPT11]{CPT} F. Campana, T. Peternell,{\em Geometric stability of the cotangent bundle and the universal cover of a projective manifold (with an appendix by Matei Toma).} Bull. Soc. Math. France {\bf 139}, 41–74 (2011).

\bibitem[F05]{F05} O. Fujino, {\em Remarks on algebraic fiber spaces}, J. Math. Kyoto Univ. {\bf 45-4} (2005), 683--699.

\bibitem[F17]{F17} O. Fujino, {\em Notes on the weak positivity theorems,} in "Algebraic varieties and automorphism groups", Adv. Stud. Pure Math. {\bf 75}, Math. Soc. Japan, Tokyo, 2017, pp. 73–118.

\bibitem[F20]{F20}  O. Fujino, {\em Iitaka conjecture-an introduction,} SpringerBriefs in Mathematics. Springer, Singapore (2020), 128 pp

\bibitem[FM00]{FM} O. Fujino, S. Mori, {\em A canonical bundle formula}, J. Differential Geometry {\bf 56} (2000), 167-188.

\bibitem[H20]{H} K. Hashizume, {\em Log Iitaka conjecture for abundant log canonical fibrations}, Proc. Japan. Acad. Series A, {\bf 96}(2020), 87-92.

\bibitem[HPS18]{HPS} C.D. Hacon, M. Popa, C. Schnell, {\em Algebraic fiber spaces over abelian varieties: around a recent theorem by Cao and Paun,} in Local and Global Methods in Algebraic Geometry. Contemporary Mathematics, vol. 712 (American Mathematical Society, Providence, 2018), pp. 143–195

\bibitem[K81]{K81}  Y. Kawamata, {\em Characterization of abelian varieties}, Compositio Math. {\bf 43}(1981), 253–276.

\bibitem[K82]{K82} Y. Kawamata; {\em Kodaira dimension of algebraic fiber spaces over curves}, Invent. Math. {\bf 66} (1982), no. 1, 57–71.

\bibitem[K85]{K85} Y. Kawamata;{\em Minimal models and the Kodaira dimension of algebraic fiber spaces}, J. Reine Angew. Math. {\bf 363} (1985), 1–46.

\bibitem[K92]{Kaw2} Y. Kawamata, {\em Abundance theorem for minimal threefolds}, Invent. Math. {\bf 108} (1992), no. 2, 229–246.

\bibitem[KM98]{KM}  J. Koll\'ar and S. Mori,{\em  Birational geometry of algebraic varieties}, Cambridge Tracts in Math.,134 (1998)

\bibitem[Lai11]{L}  C-J. Lai, {\em Varieties fibered by good minimal models.} Math. Ann. Vol. {\bf 350}, Issue 3 (2011), 533-547.

\bibitem[Laz04]{Laz} R.Lazarsfeld, {\em Positivity in algebraic geometry,} Vols. I and II, Springer, New York, 2004.

\bibitem[LP20]{LP} V. Lazi\'c, T. Peternell, {\em On generalised abundance, I}, Publ. Res. Inst. Math. Sci., {\bf 56} (2020), no. 2, 353–389.

\bibitem[N04]{N} N. Nakayama, {\em Zariski-decomposition and abundance,} MSJ Memoirs, vol. 14, Mathematical Society of Japan, Tokyo, 2004.

\bibitem[PS22]{PS} M. Popa and C. Schnell, {\em On the behavior of Kodaira dimension under smooth morphisms,} preprint, arXiv:2202.02825 (2022).

\bibitem[S95]{S} F. Serrano, {\em Strictly nef divisors and Fano threefolds,} J. Reine Angew. Math. {\bf 464} (1995), 187–206.

\bibitem[T01]{T} H. Tsuji; {\em Numerical trivial fibrations}. arXiv:math/0001023v6.

\bibitem[U75]{U} K. Ueno, {\em Classification theory of algebraic varieties and compact complex spaces,} Springer-Verlag, Berlin-New York, 1975. Notes written in collaboration with P. Cherenack, Lecture Notes in Mathematics, Vol. {\bf 439}.

\bibitem[V83]{V} E. Viehweg, {\em Weak positivity and the additivity of the Kodaira dimension for certain fiber spaces}, in Algebraic varieties and analytic varieties (Mathematical Society of Japan, Tokyo, 1983), 329–353. 

\bibitem[WZ21]{WZ} J. Wang and G. Zhong, {\em Strictly nef divisors on singular threefolds}, preprint, arxiv:2112.03117.

\end{thebibliography}
\end{document}